\documentclass[11pt,reqno,a4paper]{amsart}%
\usepackage[a4paper,margin=3cm]{geometry}%
\usepackage[pagebackref]{hyperref}%
\usepackage{bbm,lmodern,mathtools,amssymb}%
\usepackage[utf8]{inputenc}%
\usepackage[T1]{fontenc}%

\title[Poincaré and log-Sobolev for singular Gibbs measures]%
{On Poincaré and logarithmic Sobolev inequalities\\ %
  for a class of singular Gibbs measures}

\author{Djalil Chafaï} 
\address[DC]{CEREMADE, CNRS UMR 7534, Université
  Paris-Dauphine, PSL, France.}

\email{\url{mailto:djalil(at)chafai.net}} %
\urladdr{\url{http://djalil.chafai.net/}}

\author{Joseph Lehec}
\address[JL]{CEREMADE, CNRS UMR 7534, Université
  Paris-Dauphine and DMA, CNRS UMR 8553, \'Ecole Normale Supérieure, France.}

\email{\url{mailto:lehec(at)ceremade.dauphine.fr}}
\urladdr{\url{https://www.ceremade.dauphine.fr/~lehec/}}

\date{Spring 2019. Submitted to Geometric Aspects of Functional Analysis (GAFA) Seminar notes.}

\newtheorem{theorem}{Theorem}[section]%
\newtheorem{corollary}[theorem]{Corollary}%
\newtheorem{lemma}[theorem]{Lemma}%
\theoremstyle{definition}
\newtheorem{remark}[theorem]{Remark}%
\newcommand{\LIP}[1]{\|#1\|_{\mathrm{Lip}}} 
\newcommand{\R}{\mathbb R}
\newcommand{\e}{\mathrm e}
\renewcommand{\d}{\mathrm d}
\newcommand{\E}{\mathbb E}
\makeatletter
\def\@MRExtract#1 #2!{#1}     
\renewcommand{\MR}[1]{
  \xdef\@MRSTRIP{\@MRExtract#1 !}%
  \href{http://www.ams.org/mathscinet-getitem?mr=\@MRSTRIP}{MR-\@MRSTRIP}}
\makeatother

\numberwithin{equation}{section}

\keywords{Boltzmann--Gibbs measure; Gaussian unitary ensemble; Random
  matrix theory; Spectral analysis; Geometric functional analysis; Log-concave
  measure; Poincaré inequality; Logarithmic Sobolev inequality; Concentration
  of measure; Diffusion operator; Orthogonal polynomials.}

\subjclass[2000]{33C50 (pri.) and 39B62; 46N55; 46E35; 60B20; 60E15; 60J60}

\begin{document}

\begin{abstract}
  This note, mostly expository, is devoted to Poincaré and log-Sobolev
  inequalities for a class of Boltzmann--Gibbs measures with singular
  interaction. Such measures allow to model one-dimensional particles with
  confinement and singular pair interaction. The functional inequalities come
  from convexity. We prove and characterize optimality in the case of
  quadratic confinement via a factorization of the measure. This optimality
  phenomenon holds for all beta Hermite ensembles including the Gaussian
  unitary ensemble, a famous exactly solvable model of random matrix theory.
  We further explore exact solvability by reviewing the relation to
  Dyson--Ornstein--Uhlenbeck diffusion dynamics admitting the
  Hermite--Lassalle orthogonal polynomials as a complete set of
  eigenfunctions. We also discuss the consequence of the log-Sobolev
  inequality in terms of concentration of measure for Lipschitz functions such
  as maxima and linear statistics.
\end{abstract}

\maketitle

{\footnotesize\tableofcontents}

\section{Introduction}

The aim of this note is first to provide synthetic exposition gathering
material from several distant sources, and second to provide extensions and
novelty about optimality.

Let $n\in\{1,2,\ldots\}$. For a given $\rho\in \R$, we say that a function
$\phi \colon \R^n \to \R\cup\{+\infty\}$ is $\rho$-convex when
$x\mapsto \phi (x) - \rho \vert x \vert^2 /2$ is convex, where
$\vert x\vert:=\sqrt{x_1^2+\cdots+x_n^2}$ is the Euclidean norm. In particular
a $0$-convex function is just a convex function. An equivalent condition is
that for all $x,y\in\mathbb{R}^n$ and $\lambda \in [0,1]$,
\[
\phi ( (1-\lambda) x + \lambda y ) \leq (1-\lambda) \phi ( x) + \lambda \phi (y) 
- \frac{\rho \, \lambda(1-\lambda) } 2 \vert y - x \vert^2.
\]
If $\phi$ is $\mathcal C^2$-smooth on its domain then this is yet equivalent
to $\mathrm{Hess}(\phi) \geq \rho \, I_n$ as quadratic forms, pointwise, where
$I_n$ is the identity:
$\langle\mathrm{Hess}(\varphi)(x)y,y\rangle\geq\rho\vert x\vert^2$ for all
$x,y\in\R^n$.

Let $V:\mathbb{R}^n\to\mathbb{R}$ and
$W:\mathbb{R}\to\mathbb{R}\cup\{+\infty\}$ be two functions, called
``confinement potential'' and ``interaction potential'' respectively. We
assume that
\begin{itemize}
\item $V$ is $\rho$-convex for some $\rho>0$;
\item $W$ is convex with domain $(0;+\infty)$. In particular $W\equiv +\infty$
  on $(-\infty ; 0]$.
\end{itemize}
The energy of a configuration $x=(x_1,\ldots,x_n)\in\mathbb{R}^n$ is
\begin{equation*}\label{eq:U}
  U(x) =  V(x_1,\ldots,x_n) + \sum_{i<j}W(x_i-x_j)
  =V(x)+U_W(x)
  \in\mathbb{R}\cup\{+\infty\}.
\end{equation*}
The nature of $W$ gives that $U(x)$ is finite if and only if $x$ belongs to
the ``Weyl chamber''
\[
D=\{x\in\mathbb{R}^n:x_1>\cdots>x_n\}. 
\]
Assuming that 
\begin{equation*}\label{eq:Z}
  Z_{\mu}=\int_{\mathbb{R}^n}\mathrm{e}^{- U(x_1,\ldots,x_n)}\d x_1\cdots\d x_n
  <\infty
\end{equation*}
we define a probability measure $\mu$ on $\mathbb{R}^n$ by
\begin{equation}\label{eq:mun}
 \mu(\mathrm d x)
  =\frac{\mathrm{e}^{-U(x_1,\ldots,x_n)}}{Z_{\mu}}\d x.
\end{equation}
The support of $\mu$ is
$\overline{D}=\{x\in\mathbb{R}^n:x_1\geq\cdots\geq x_n\}$. Note that if
\begin{equation}\label{eq:Wlog}
W(u)=
\begin{cases}
- \beta\log u , & \text{if } u > 0 \\
+\infty & \text{otherwise}
\end{cases}
\end{equation}
where $\beta$ is a positive parameter, and if $X$ is a
random vector of $\mathbb{R}^n$ distributed according to $\mu$, then for every 
$\sigma>0$, the scaled random vector $\sigma X$ follows the law
$\mu$ with same $W$ but with $V$ replaced by $V(\cdot/\sigma)$.

Following Edelman \cite{MR1936554}, the beta Hermite ensemble corresponds to the case
\[
  V(x)=\frac{n}{2}\vert x\vert^2 = \frac n2 ( x_1^2 + \dotsb + x_n^2 ) ,
\]
and $W$ given by~\eqref{eq:Wlog}. In this case $\mu$ rewrites using a
Vandermonde determinant as
\begin{equation}\label{eq:BHE}
  \d\mu(x)
  =\frac{\mathrm{e}^{-\frac{n}{2}|x|^2}}{Z_{\mu}}\prod_{i<j}(x_i-x_j)^\beta 
\, \mathbbm 1_{\{x_1 \geq \dotsb \geq x_n \}} 
  \d x.
\end{equation}
The normalizing constant $Z_{\mu}$ can be explicitly computed in terms of
Gamma functions by reduction to a classical Selberg integral, but this is
useless for our purposes in this work. The Gaussian unitary ensemble (GUE) of
Dyson \cite{MR0177643} corresponds to $\beta=2$, namely
\begin{equation}\label{eq:GUE}
  \d\mu(x)
  =\frac{\mathrm{e}^{-\frac{n}{2}|x|^2}}{Z_{\mu}}\prod_{i<j}(x_i-x_j)^2
\mathbbm 1_{\{x_1 \geq \dotsb \geq x_n \}} 
  \d x.
\end{equation}
Note that on $\mathbb{R}^n$ the density of the beta Hermite ensemble
\eqref{eq:BHE} with respect to the Gaussian law
$\mathcal{N}(0,\frac{1}{n}I_n)$ is equal up to a multiplicative constant to
$\prod_{i<j}(x_i-x_j)^\beta$ times the indicator function of the Weyl chamber.
The cases $\beta=1$ and $\beta=4$ are known as the Gaussian orthogonal
ensemble (GOE) and the Gaussian simplectic ensemble (GSE).

Let $L^2(\mu)$ be the Lebesgue space of measurable functions from
$\mathbb{R}^n$ to $\mathbb{R}$ which are square integrable with respect to
$\mu$. Let $H^1(\mu)$ be the Sobolev space of functions in $L^2(\mu)$
with weak derivative in $L^2(\mu)$ in the sense of Schwartz distributions.

We provide in Section \ref{se:facts} some useful or beautiful facts about
\eqref{eq:mun}, \eqref{eq:BHE}, and \eqref{eq:GUE}.

\subsection{Functional inequalities and concentration of measure}

Given $f\in L^2 ( \mu)$ we define the variance of $f$ with respect to $\mu$ by
\[
\mathrm{var}_\mu ( f )  = \int_{\R^n} f^2 \, \mathrm d \mu 
- \left( \int_{\R^n} f \, \mathrm d \mu \right)^2.
\]
If additionally $f \geq 0$, then we define similarly the entropy of $f$ with respect to
$\mu$ by
\[
\mathrm{ent}_\mu (f) = 
\int_{\R^n}  f \log f \, \d \mu 
- \left( \int_{\R^n}  f \, \d \mu \right) \log\left( \int_{\R^n} f \, \d \mu \right).
\]
\begin{theorem}[Poincaré inequality]\label{th:PI}
  Let $\mu$ be as in \eqref{eq:mun}. For all $f\in H^1(\mu)$,
  \[
    \mathrm{var}_{\mu}(f) 
    \leq \frac{1}{\rho}\int_{\R^n} |\nabla f|^2 \, \d \mu .
  \]
  This holds in particular with $\rho=n$ for the beta Hermite ensemble
  \eqref{eq:BHE} for all $\beta>0$.
\end{theorem}

\begin{theorem}[Log-Sobolev inequality]\label{th:LSI}
  Let $\mu$ be as in \eqref{eq:mun}. For all $f\in H^1(\mu)$,
  \[
    \mathrm{ent}_\mu ( f^2 ) %
    \leq \frac{2}{\rho}\int_{\R^n} |\nabla f|^2 \, \d \mu.
  \]
  This holds in particular with $\rho=n$ for the beta Hermite ensemble
  \eqref{eq:BHE} for all $\beta>0$.
\end{theorem}

\begin{theorem}[Optimality for Poincaré and log-Sobolev
  inequalities]\label{th:opt}
  Let $\mu$ be as in \eqref{eq:mun}. Assume that $V$ is quadratic:
  $V(x)= \rho\vert x\vert^2/2$ for some $\rho>0$. This is in particular the
  case for the beta Hermite ensemble \eqref{eq:BHE} for all $\beta>0$. Then
  equality is achieved in the Poincaré inequality of Theorem \ref{th:PI} for
  \[
    f:x\in\mathbb{R}^n\mapsto\lambda(x_1+\cdots+x_n)+c, %
    \quad \lambda,c\in\mathbb{R}. 
  \]
  Moreover equality is achieved in the logarithmic Sobolev inequality of Theorem
  \ref{th:LSI} for
  \[
    f:x\in\mathbb{R}^n\mapsto\mathrm{e}^{\lambda(x_1+\cdots+x_n)+c}, %
    \quad\lambda,c\in\mathbb{R}.
  \]
  Lastly, in both cases these are the only extremal functions.   
\end{theorem}

Theorems \ref{th:PI} and \ref{th:LSI} are proved in Section
\ref{ss:proof:pilsi} and Theorem \ref{th:opt} in Section \ref{ss:proof:opt}.

Poincaré and logarithmic Sobolev inequalities for beta ensembles are already
known in the literature about random matrix theory, see for instance
\cite{MR2760897,MR3699468} and references therein. However the optimality that
we point out here seems to be new.

The following corollary of Theorem \ref{th:LSI} provides concentration of
measure around the mean for Lipschitz functions, including linear statistics
and maximum.

\begin{corollary}[Gaussian concentration inequality for Lipschitz
  functions]\label{co:conc}
  Let $\mu$ be as in \eqref{eq:mun}. For every Lipschitz function
  $F:\mathbb{R}^n\to\mathbb{R}$ and for all real parameter $r>0$,
  \begin{equation}\label{eq:conc}
    \mu\left(\left|F-\int F\d\mu\right|\geq r\right)
    \leq2\exp\left(-\frac{\rho}{\LIP{F}^2}\frac{r^2}{2}\right).
  \end{equation}
  In particular for any measurable $f:\mathbb{R}\to\mathbb{R}$ and all $r>0$,
  with $L_n(f)(x)=\frac{1}{n}\sum_{i=1}^nf(x_i)$,
  \begin{equation}\label{eq:conclinstat}
    \mu
    \left(\left|L_n(f)-\int L_n(f)\d\mu\right|\geq r\right) %
    \leq2\exp\left(-n\frac{\rho}{\LIP{f}^2}\frac{r^2}{2}\right).
  \end{equation}
  Additionally, for all $r>0$,
  \begin{equation}\label{eq:concmax}
	  \mu\left(\Bigr|x_1-\int x_1\mu(\d x)\Bigr|\geq r\right)
    \leq2\exp\left(-\rho\frac{r^2}{2}\right).
  \end{equation}
  This holds in particular with $\rho=n$ for the beta Hermite ensemble
  \eqref{eq:BHE} for all $\beta>0$.
\end{corollary}

The proof of Corollary \ref{co:conc} and some additional comments are given in
Section \ref{ss:proof:conc}.

The scale in \eqref{eq:concmax} is not optimal for the beta Hermite ensemble,
the largest particle is actually more concentrated than what is predicted by
Corollary~\ref{co:conc}. Indeed, it is proved for instance in \cite{MR2813333}
that $n^{2/3} (\lambda_1 - 2)$ converges in law as $n$ tends to infinity to a
Tracy--Widom distribution of parameter $\beta$. In particular fluctuations of
$\lambda_1$ are of order $n^{-2/3}$, whereas~\eqref{eq:concmax} only predicts
an upper bound of order $n^{-1/2}$. See also
\cite{MR2126983,MR2053053,MR2678393} for a concentration property that matches
the correct order of fluctuations.

Note also that \eqref{eq:conclinstat} allows to get concentration for the
Cauchy--Stieltjes transform of the empirical distribution by taking $f$ equal
to the real or imaginary part of $x\mapsto 1/(x-z)$ where $z=a+\mathrm{i}b$ is
a fixed complex parameter with $b>0$.

The function $(x_1,\ldots,x_n)\mapsto L_n(f)(x)=\frac{1}{n}\sum_{i=1}^nf(x_i)$
is called a linear statistics. 
The inequality \eqref{eq:conclinstat} appears for the spectrum of random
matrix models in the work of Guionnet and Zeitouni \cite{MR1781846} via the
logarithmic Sobolev inequality, see also \cite{MR2498298} and \cite[Section
4.4.2]{MR2760897}, and \cite[Exercise\ 4.4.33]{MR2760897} for beta ensembles.
The exponential speed $n^2$ in \eqref{eq:conclinstat} is optimal according to
the large deviation principle satisfied by $L_n$ under $\mu$ established by
Ben Arous and Guionnet \cite{MR1465640} for the GUE, see \cite{MR3262506} and
references therein for the general case \eqref{eq:mun}. 
Concentration inequalities
and logarithmic Sobolev inequalities for spectra of some random matrix models
at the correct scale can also be obtained using coupling methods or exact
decompositions, see for instance \cite{MR3055255,MR3109633} and references
therein.

Many proofs involve the following simple transportation facts:
\[
  \mathcal{N}(0,n^{-1}I_n)
  \overset{\mathrm{Caffarelli}}{\xrightarrow{\hspace*{4em}}}
  \mu
  \overset{x_1+\cdots+x_n}{\xrightarrow{\hspace*{4em}}}
  \mathcal{N}(0,1)
\]
and 
\[
  \mathcal{N}(0,n^{-1}I_n) 
  \overset{x_1+\cdots+x_n}{\xrightarrow{\hspace*{4em}}}
  \mathcal{N}(0,1)
\]
and
\[
  \mathrm{Law}(H)
  \overset{\mathrm{Spectrum}}{\xrightarrow{\hspace*{4em}}}
  \mu
  \overset{x_1+\cdots+x_n}{\xrightarrow{\hspace*{4em}}}
  \mathcal{N}(0,1)
\]   
and
\[
  \mathrm{Law}(H)
  \overset{\mathrm{Trace}}{\xrightarrow{\hspace*{4em}}}
  \mathcal{N}(0,1)
\]
where $H$ is a random Hermitian matrix as in Theorem \ref{th:GUEM} or
Theorem \ref{th:BHEM}.

\subsection{Dynamics}
\label{se:dynamics}

Let us assume in this section that the functions $V$ and $W$ are smooth on
$\R^n$ and $(0,+\infty)$ respectively. Then the energy $U$ is smooth on its
domain $D$. Fix $X_0\in D$ and consider the overdamped Langevin diffusion
associated to the potential $U$ starting from $X_0$, solving the stochastic
differential equation
\begin{equation}\label{eq:SDE}
  X_t = X_0 + \sqrt{2} B_t- \int_0^t \nabla U(X_s)\, \d s + \Phi_t,
  \ t\geq0,
\end{equation}
where ${(B_t)}_{t\geq0}$ is a standard Brownian Motion of $\mathbb{R}^n$, and
where $\Phi_t$ is a reflection at the boundary of $D$ which constrains the
process $X$ to stay in $D$. More precisely
\[
  \Phi_t = - \int_0^t\mathrm{n}_s\, L(\d s)
\]
where $L$ is a random measure depending on $X$ and supported on
$\{t \geq 0 \colon X_t \in \partial D\}$ and where
$\mathrm{n}_t$ is an outer unit normal to the boundary of $D$ at $X_t$ for
every $t$ in the support of $L$. The process $L$ is called the 
``local time'' at the boundary of $D$. The stochastic differential equation
\eqref{eq:SDE} writes equivalently
\[
  \d X_t=\sqrt{2}\, \d B_t-\nabla U(X_t)\,\d t -\mathrm{n}_t \, L (\d t).
\]
It is not obvious that equation~\eqref{eq:SDE} admits a solution. Such
diffusions with reflecting boundary conditions were first considered by
Tanaka. He proved in \cite{MR529332} that if $\nabla U$ is globally Lipschitz
on $D$ and grows at most linearly at infinity then~\eqref{eq:SDE} does admit a
unique strong solution.

If it exists, the solution is a Markov process. Its generator is the operator
$G$ where
\begin{equation}\label{eq:diff}
  \mathrm{G}
  =\Delta-\langle\nabla U,\nabla\rangle
  =\sum_{i=1}^n\partial^2_{x_i}
  -\sum_{i=1}^n (\partial_{x_i}V)(x)\partial_{x_i}
  -\sum_{i\neq j} W'(x_i-x_j) \partial_{x_i}
\end{equation}  
with Neumann boundary conditions at the boundary of $D$.
Stokes formula then shows that $\mathrm{G}$ is
symmetric in $L^2 (\mu)$. As a result the measure $\mu$ is reversible for the
process $(X_t)$. By integration by parts the density $f_t$ of $X_t$ with
respect to the Lebesgue measure satisfies the Fokker--Planck equation
$\partial_tf_t=\Delta f_t+\mathrm{div}(f_t\nabla U)$.

It is common to denote $X_t=(X_t^1,\ldots,X_t^n)$ and to interpret
$X_t^1,\ldots,X_t^n$ as interacting particles on the real line experiencing
confinement and pairwise interactions. Let us discuss now the particular case
of the beta Hermite ensemble \eqref{eq:BHE}, for which \eqref{eq:SDE} rewrites
\begin{equation}\label{eq:SDE2}
  \d X_t^i = \sqrt{2} \,\d B_t^i - nX_t^i\d t 
  + \beta\sum_{j\colon j\neq i} \frac 1 {X^i_t - X^j_t} \,\d t  
  , \quad 1\leq i \leq n 
\end{equation}
as long as the particles have not collided. We call this diffusion the
Dyson--Ornstein--Uhlenbeck process. Without the confinement term
$-X_t^i\, \d t$ this diffusion is known in the literature as the Dyson
Brownian motion. Indeed Dyson proved in \cite{MR0148397} the following
remarkable fact: if $(M_t)$ is an Ornstein--Uhlenbeck process taking values in
the space of complex Hermitian matrices then the eigenvalues of $(M_t)$ follow
the diffusion~\eqref{eq:SDE2} with parameter $\beta = 2$, while if $(M_t)$ is
an Ornstein--Uhlenbeck process taking values in the space of real symmetric
matrices then the same holds true with $\beta = 1$. Dyson also proved an
analogue result for the eigenvalues of a Brownian motion on the unitary group.
It is natural to ask whether the repulsion term $1 / (X_t^i - X_t^j)$ is
strong enough to actually prevent the collision of particles. This was
investigated by Rogers and Shi in \cite{MR1217451}, see also \cite{MR2760897}.
They proved that if $\beta \geq 1$ then there are no collisions:
\eqref{eq:SDE2} admits a unique strong solution and with probability $1$, the
process $(X_t)$ stays in the Weyl chamber $D$ for all time. This means that in
that case, Tanaka's equation~\eqref{eq:SDE} does admit a unique strong
solution, but the reflection at the boundary $\Phi_t$ is actually identically
$0$. This critical phenomenon was also observed twenty five years ago by
Calogero in \cite{MR0280103}.
Besides, although it is not explicetly written in Rogers and Shi's article, 
when $\beta < 1$ collisions do occur in finite time, so
that the reflection $\Phi_t$ enters the picture. In that case though, the
existence of a process $(X_t)$ satisfying~\eqref{eq:SDE} does not follow from
Tanaka's theorem~\cite{MR529332}, as the potential $U$ is singular at the
boundary of $D$. Still~\eqref{eq:SDE} does admit a unique strong solution.
Indeed, this was established by Cépa and Lépingle in \cite{MR1440140} using an
existence result for multivalued stochastic differential equations due to
Cépa~\cite{MR1459451}. See also the work of Demni \cite{MR2470522,demni}.

\emph{Long time behavior of the dynamics.} Let us assume that the
process~\eqref{eq:SDE} is well defined. We denote by $(P_t)$ the associated
semigroup: For every test function $f$
\[
P_t f (x)  = \E(f ( X_t ) \mid X_0 = x).
\]
Given a probability measure $\nu$ on $\R^n$ we denote $\nu P_t$ the law of the
process at time $t$ when initiated from $\nu$. Recall that the measure $\mu$
is stationary: $\mu P_t = \mu$ for all time. For all real number $p\geq 1$,
the $L^p$ Kantorovich or Wasserstein distance between $\mu$ and $\nu$ is
\begin{equation}\label{eq:Wp}
  \mathrm{W}_p ( \nu , \mu ) 
  = \inf_{\substack{(X,Y)\\X\sim\nu\\Y\sim\mu}}\E(\vert X - Y \vert^p)^{1/p}.
\end{equation}
Note that $\mathrm{W}_p(\nu,\mu)<\infty$ if
$\left|\cdot\right|^p\in L^1(\nu)\cap L^1(\mu)$. It can be shown that the
convergence for $\mathrm{W}_p$ is equivalent to weak convergence together with
convergence of $p$-th moment. If $\nu$ has density $f$ with respect to $\mu$, the
relative entropy of $\nu$ with respect to $\mu$ is
\begin{equation}\label{eq:H}
  \mathrm H ( \nu \mid \mu )  = \int_{\R^n} \log f \, d\nu .
\end{equation}
If $\nu$ is not absolutely continuous we set
$\mathrm H ( \nu \mid \mu ) = +\infty$ by convention.

\begin{theorem}[Convergence to equilibrium]\label{th:equilib}
  For any two probability measures $\nu_0, \nu_1$ on $\R^n$ we have, for all
  $p\geq1$ and $t\geq0$, in $[0,+\infty]$,
  \[
    \mathrm{W}_p ( \nu_0 P_t , \nu_1 P_t ) %
    \leq \e^{- \rho t} \, \mathrm{W}_p ( \nu_0 , \nu_1 ) . 
  \]
  In particular, choosing $\nu_1 = \mu$ yields  
  \[
    \mathrm{W}_p ( \nu_0 P_t , \mu ) %
    \leq \e^{- \rho t} \, \mathrm{W}_p ( \nu_0 , \mu ).
  \]
  Moreover we also have, for all $t\geq0$,
  \begin{equation}\label{eq:relent}
    \mathrm H ( \nu_0 P_t \mid \mu ) %
    \leq \e^{-\rho t }\, \mathrm H ( \nu_0 \mid \mu ).
  \end{equation}
\end{theorem}

A proof of Theorem \ref{th:equilib} is given in Section \ref{ss:proof:equilib}.
\subsection{Hermite--Lassalle orthogonal polynomials}
\label{se:polys}

Recall that for all $n\geq1$, the classical Hermite polynomials
${(H_{k_1,\ldots,k_n})}_{k_1\geq0,\ldots,k_n\geq0}$ are the orthogonal
polynomials for the standard Gaussian distribution $\gamma_n$ on $\R^n$. The
tensor product structure $\gamma_n=\gamma_1^{\otimes n}$ gives
$H_{k_1,\ldots,k_n}(x_1,\ldots,x_n)=H_{k_1}(x_1)\cdots H_{k_n}(x_n)$ where
${(H_k)}_{k\geq0}$ are the orthogonal polynomials for the one-dimensional
Gaussian distribution $\gamma_1$. Among several remarkable characteristic
properties, these polynomials satisfy a differential equation which writes
\begin{equation}\label{eq:OU}
  \mathrm{L}H_{k_1,\ldots,k_n}=-(k_1+\cdots+k_n)H_{k_1,\ldots,k_n}
  \quad\text{where}\quad
  \mathrm{L}=\Delta-\langle x,\nabla\rangle
\end{equation}
is the infinitesimal generator of the Ornstein--Uhlenbeck process, which
admits $\gamma_n$ as a reversible invariant measure. In other words these
orthogonal polynomials form a complete set of eigenfunctions of this operator.
Such a structure is relatively rare, see \cite{MR1478714} for a complete
classification when $n=1$.

Lassalle discovered in the 1990s that a very similar phenomenon takes place
for beta Hermite ensembles and the Dyson--Ornstein--Uhlenbeck process,
provided that we restrict to symmetric polynomials. Observe first that this
cannot hold for all polynomials, simply because the infinitesimal generator
\begin{equation}\label{eq:diff-beta}
  \mathrm{G}
  =\sum_{i=1}^n\partial^2_{x_i}
    -n\sum_{i=1}^nx_i\partial_{x_i}
    +\beta\sum_{i\neq j}\frac{1}{x_i-x_j}\partial_{x_i},
\end{equation}  
of the Dyson--Ornstein--Uhlenbeck process, which is a special case of
\eqref{eq:diff}, does not preserve polynomials, for instance we have
$\mathrm{G}x_1=-nx_1+\beta\sum_{j\neq 1}\frac{1}{x_1-x_j}$. However, rewriting
this operator by symmetrization as
\begin{equation}\label{eq:diff-beta-s}
  \mathrm{G}
  =\sum_{i=1}^n\partial^2_{x_i}
  -n\sum_{i=1}^nx_i\partial_{x_i}
  +\frac{\beta}{2}\sum_{i\neq j}\frac{1}{x_i-x_j}(\partial_{x_i}-\partial_{x_j}) ,
\end{equation}  
it is easily seen that the set of symmetric polynomials in $n$ variables is left
invariant by $\mathrm{G}$.

Let $\mu$ be the beta Hermite ensemble defined in \eqref{eq:BHE}.
Lassalle studied in
\cite{MR1133488} multivariate symmetric polynomials
${(P_{k_1,\ldots,k_n})}_{k_1\geq\cdots\geq k_n\geq0}$ which are orthogonal
with respect to $\mu$. He called them ``generalized Hermite'' but we
decide to call them ``Hermite--Lassalle''. For all
$k_1\geq\cdots\geq k_1\geq0$ and $k_1'\geq\cdots\geq k_n'\geq0$,
\begin{equation}\label{eq:orth}
  \int P_{k_1,\ldots,k_n}(x_1,\ldots,x_n)P_{k_1',\ldots,k_n'}(x_1,\ldots,x_n)
  \mu(\d x)
  =\mathbbm{1}_{(k_1,\ldots,k_n)=(k_1',\ldots,k_n')}.
\end{equation}
They can be obtained from the standard basis of symmetric polynomials by using
the Gram--Schmidt algorithm in the Hilbert space $L^2_{\mathrm{sym}}(\mu)$ of
square integrable symmetric functions. The total degree of
$P_{k_1,\ldots,k_n}$ is $k_1+\cdots+k_n$, in particular $P_{0,\dotsc,0}$ is a
constant polynomial. The numbering in terms of $k_1,\ldots,k_n$ used in
\cite{MR1133488} is related to Jack polynomials. Beware that \cite{MR1133488}
comes without proofs. We refer to \cite{MR1471336} for proofs, and to
\cite{MR2325917} for symbolic computation via Jack polynomials. 

The Hermite--Lassalle symmetric polynomials form an orthogonal basis in
$L^2_{\mathrm{sym}}(\mu)$ of eigenfunctions of the Dyson--Ornstein--Uhlenbeck
operator $\mathrm G$. Restricted to symmetric functions, this operator 
is thus exactly solvable, just like the classical
Ornstein--Uhlenbeck operator. Here is the result of Lassalle in
\cite{MR1133488}, see \cite{MR1471336} for a proof.

\begin{theorem}[Eigenfunctions and eigenvalues]\label{th:lassalle:generator}
  For all $n\geq2$ and $k_1\geq\cdots\geq k_n\geq0$,
  \begin{equation}\label{eq:eigfunc}
    \mathrm{G}P_{k_1,\ldots,k_n}=-n(k_1+\cdots+k_n)P_{k_1,\ldots,k_n}.
  \end{equation}
  where $\mathrm{G}$ is the operator \eqref{eq:diff-beta}.
\end{theorem}

When $\beta=0$ then $G$ becomes the Ornstein--Uhlenbeck operator. For all
$\beta>0$, the spectrum of $G$ is identical to the one of the
Ornstein--Uhlenbeck operator.
This can be guessed from the fact that the eigenfunctions
are polynomials together with the fact that the interaction term (the non
O.-U.\ part) lowers the degree of polynomials.

The spectral gap of $\mathrm{G}$ in $L^2_{\mathrm{sym}}(\mu)$ is $n$: if
$f\in L^2_\mathrm{sym} (\mu)$ is orthogonal to constants then
\[
n \int f^2 \, d\mu \leq - \int f\, \mathrm{G}f \, \d\mu = \int \vert \nabla f \vert^2 \, d\mu .
\]
Theorem~\ref{th:PI} shows that this inequality holds actually for all $f$, not
only symmetric ones.

Hermite--Lassalle polynomials can be decomposed in terms of Jack
polynomials, and this decomposition generalizes the hypergeometric expansion
of classical Hermite polynomials.

\begin{remark}[Examples and formulas]\label{rm:xpl}
  It is not difficult to check that up to normalization
  \[
    x_1+\cdots+x_n
    \quad\text{and}\quad
    x_1^2+\cdots+x_n^2-1-\beta\frac{n-1}{2}.
    \
  \]
  are Hermite--Lassalle polynomials. In the GUE case, $\beta=2$, Lassalle
  gave in \cite{MR1133488}, using Jack polynomials and Schur functions, a
  formula for $P_{k_1,\ldots,k_n}$ in terms of a ratio of a determinant
  involving classical Hermite polynomials and a Vandermonde determinant.
\end{remark}

\subsection{Comments and open questions}

Regarding functional inequalities, one can probably extend the results to the
class of Gaussian $\varphi$-Sobolev inequalities such as the Beckner
inequality \cite{MR954373}, see also \cite{MR2081075}. Lassalle has studied
not only the beta Hermite ensemble in \cite{MR1133488}, but also the
beta Laguerre ensemble in \cite{MR1105634} with density proportional to
\[
  x\in D\mapsto \prod_{k=1}^nx_k^a\mathrm{e}^{-bnx_k}
  \prod_{i<j}(x_i-x_j)^\beta \mathbbm{1}_{x_1\geq\cdots\geq x_n\geq0},
\]
and the beta Jacobi ensemble in \cite{MR1096625} with density proportional to
\[
  x\in
  D\mapsto\prod_{k=1}^nx_k^{a-1}(1-x_k)^{b-1}
  \prod_{i<j}(x_i-x_j)^\beta\mathbbm{1}_{1\geq x_1\geq\cdots\geq x_n\geq0}.
\]
It is tempting to study functional inequalities and concentration of measure
for these ensembles. The proofs of Lassalle, based on Jack polynomials, are
not in \cite{MR1105634,MR1096625,MR1133488} but can be found in the work
\cite{MR1471336} by Baker and Forrester. We refer to \cite{MR1843558} for the
link with Macdonald polynomials. It is natural (maybe naive) to ask about
direct proofs of these results without using Jack polynomials. The study of
beta ensembles can be connected to $H$-transforms and to the work
\cite{MR1678525} on Brownian motion in a Weyl chamber, see also
\cite{doumerc}. The analogue of the Dyson Brownian motion for the Laguerre
ensemble is studied in \cite{MR1132135}, see also
\cite{MR1871699,doumerc,MR3769666,voit-woerner}. Tridiagonal matrix models for
Dyson Brownian motion are studied in \cite{holcomb-paquette}.

The natural isometry between $L^2(\gamma_n)$ and $L^2(\d x)$ leads to
associate to the Ornstein--Uhlenbeck operator a real Schrödinger operator
which turns out to be the quantum harmonic oscillator. Similarly, the natural
isometry between $L^2_{\mathrm{sym}}(\mu)$ and $L^2_{\mathrm{sym}}(\d x)$
leads to associate to the Dyson--Ornstein--Uhlenbeck operator a real
Schödinger operator known as the Calogero--Moser--Sutherland operator, which
is related to radial Dunkl operators, see for instance
\cite{MR1652182,MR1843558}. The fact that the eigenfunctions of such operators
are explicit and involve polynomials goes back at least to Calogero
\cite{MR0280103}, more than twenty five years before Lassalle!

The factorization phenomenon captured by Lemma \ref{le:factor}, which is
behind the optimality provided by Theorem \ref{th:opt}, reminds some kind of
concentration-compactness related to continuous spins systems as in
\cite{MR2028218} and \cite{MR1847094} for instance. The factorization Lemma
\ref{le:factor} remains valid for other ensembles such as the Beta-Ginibre
ensemble with density proportional to
\begin{equation}\label{eq:BGE}
  z\in\mathbb{C}^n\mapsto\mathrm{e}^{-n\sum_{k=1}^n|z_i|^2}\prod_{j<k}|z_j-z_k|^\beta,
\end{equation}
see \cite[Remark 5.4]{coulsim} for the case $n=\beta=2$. However, in contrast
with the Beta-Hermite ensemble, the interaction term is not convex in the complex case, 
and it is not clear at all what are the Poincaré and log-Sobolev constants of the Ginibre ensemble. 
See~\cite{coulsim} for an upper bound and further discussions on the associated dynamics. 

\section{Useful or beautiful facts}\label{se:facts}

\subsection{Random matrices, GUE, and beta Hermite ensemble}
\label{se:GUE}
The following result from random matrix theory goes back to Dyson, see
\cite{MR0177643,MR2129906,MR2760897,MR2641363}.

\begin{theorem}[Gaussian random matrices and GUE]\label{th:GUEM}
  The Gaussian unitary ensemble $\mu$ defined by \eqref{eq:GUE} is the law
  of the ordered eigenvalues of a random $n\times n$ Hermitian matrix $H$ with
  density proportional to
  $ h\mapsto \mathrm{e}^{-\frac{n}{2}\mathrm{Trace}(h^2)}
  =\mathrm{e}^{-\frac{n}{2}\sum_{i=1}^n h_{ii}^2 -n\sum_{i<j}|h_{ij}|^2} $ in
  other words the $n^2$ real random variables
  ${\{H_{ii},\Re H_{ij},\Im H_{ij}\}}_{1\leq i < j\leq n}$ are independent,
  with $\Re H_{ij}$ and $\Im H_{ij}$ $\sim\mathcal{N}(0,1/(2n))$ for any $i<j$
  and $H_{ii}$ $\sim\mathcal{N}(0,1/n)$ for any $1\leq i\leq n$.
\end{theorem}

There is an analogue theorem for the GOE case $\beta=1$ with random Gaussian
real symmetric matrices, and for the GSE case $\beta=4$ with random Gaussian
quaternion selfdual matrices. The following result holds for all beta Hermite
ensemble \eqref{eq:BHE}, see \cite{MR1936554}.

\begin{theorem}[Tridiagonal random matrix model for beta Hermite ensemble]
  \label{th:BHEM}
  The beta Hermite ensemble $\mu$ defined by \eqref{eq:BHE} is the
  distribution of the ordered eigenvalues of the random tridiagonal symmetric
  $n\times n$ matrix
  \[
    H=\frac{1}{\sqrt{2n}}
    \begin{pmatrix}
      \mathcal{N}(0, 2) & \chi_{(n-1) \beta} & & & \\
      \chi_{(n-1) \beta} & \mathcal{N}(0, 2)  & \chi_{(n-2) \beta} & & \\
      & \ddots & \ddots & \ddots & \\
      & & \chi_{2\beta} & \mathcal{N}(0,2) & \chi_{\beta} \\
      & &  & \chi_{\beta} & \mathcal{N}(0,2)
    \end{pmatrix}
  \]
  where, up to the scaling prefactor $1/\sqrt{2n}$, the entries in the upper
  triangle including the diagonal are independent, follow a Gaussian law
  $\mathcal{N}(0,2)$ on the diagonal, and $\chi$-laws just above the diagonal
  with a decreasing parameter with step $\beta$ from $(n-1)\beta$ to $\beta$.
\end{theorem}

In particular the trace follows the Gaussian law $\mathcal{N}(0,1)$. Such
random matrix models with independent entries allow notably to compute moments
of \eqref{eq:BHE} via traces of powers.

\subsection{Isotropy of beta Hermite ensembles}

This helps to understand the structure.
Let $\mu$ be the beta Hermite ensemble~\eqref{eq:BHE}, and let
  $\widetilde\mu$ be the probability measure
obtained from $\mu$ by symmetrizing coordinates: 
For every test function 
$f:\mathbb{R}^n\to\mathbb{R}$ we have
\[
  \int f \,\d\widetilde \mu
  =\int f_* \, \d \mu
\]
where $f_*$ is the symmetrization of $f$, defined by
\[
f_*(x_1,\ldots,x_n)=\frac{1}{n!}\sum_{\sigma\in\Sigma_n}f(x_{\sigma(1)},\ldots,x_{\sigma(n)})
\]
where $\Sigma_n$ is the symmetric group of permutations of $\{1,\ldots,n\}$.
Of course the probability measures $\mu$ and $\widetilde\mu$ coincide on
symmetric test functions. The probability measure $\widetilde \mu$ is by
definition invariant by permutation of the coordinates, and its density with
respect to the Lebesgue measure is 
\[
  \frac{\d\widetilde\mu}{\d x}
 = \frac{\mathrm{e}^{- \frac n 2 \vert x\vert^2  }}{n!Z_\mu}
  \prod_{i<j} \vert x_i - x_j\vert^\beta 
\]
Note that the support of $\widetilde \mu$ is the whole space and that $\widetilde \mu$
is not log-concave, even though $\mu$ is.
\begin{corollary}[Isotropy of beta Hermite ensemble]\label{co:isotropy}
  For every $1\leq i\neq j\leq n$,
  \[
    \int x_i \, \d\widetilde\mu =0, \quad
    \int x_i^2 \, \d\widetilde\mu = \frac \beta 2 + \frac {2-\beta}{2n} , 
    \quad \int x_ix_j \, \d\widetilde\mu  = - \frac \beta {2n} . 
  \]
  In particular, the law $\widetilde\mu$ is asymptotically isotropic. 
\end{corollary}

Recall that isotropy means zero mean and covariance matrix multiple of the
identity.

In the extremal case $\beta=0$, the measure $\widetilde\mu$ is the Gaussian
law $\mathcal{N}(0,\frac{1}{n}I_n)$.

\begin{proof}[Proof of Corollary \ref{co:isotropy}]
  Observe first that if $X \sim \mu$ then $\sum X_i$ is a standard Gaussian.
  This can be seen using Theorem~\ref{th:BHEM}, and observing that $\sum X_i$
  coincides with the trace of the matrix $H$. Actually this is true regardless
  of the interaction potential $W$, see Lemma~\ref{le:factor} below. In
  particular
  \[
    \int (x_1+\cdots+x_n) \, \mu(\d x)  = 0 ,
  \]
  hence, by definition $\widetilde\mu$,
  \[
    \int x_i \, \widetilde\mu(\d x)
    =\frac 1n \int (x_1+\dotsb+x_n) \,\mu(\d x)
    =0 , 
  \]
  for every $i\leq n$. Since $\sum X_i$ is a standard Gaussian we also have
  \begin{equation}\label{eq:isop1}
    \int (x_1+\cdots+x_n)^2 \, \mu(\d x)  = 1. 
  \end{equation}
  Next we compute $\int \vert x \vert^2\,\d\mu$. This can be done using
  Theorem~\ref{th:BHEM}, namely
  \begin{equation}\label{eq:isop2}
  \int \vert x\vert^2 \, \mu ( \d x)  =\mathbb{E}(\mathrm{Trace}(H^2))
    =1+\frac{\beta}{n}\sum_{k=1}^{n-1}k
    =1+\frac{(n-1)\beta}{2}.
  \end{equation}
  Note that the matrix model gives more: indeed, using the algebra of the
  Gamma laws,
  \[
    \mathrm{Trace}(H^2)
      \sim\mathrm{Gamma}\left(\frac{n}{2}+\frac{\beta n(n-1)}{4},\frac{n}{2}\right).
  \]

  Alternatively one can use the fact that the square of the norm
  $\left|\cdot\right|^2$ is, up to an additive constant, an eigenvector of
  $\mathrm{G}$, see Remark~\ref{rm:xpl}. Namely, recall the
  definition~\eqref{eq:diff-beta} of the operator $\mathrm G$ and note that
  \[
    \mathrm G(\left\vert\cdot\right\vert^2)(x) %
    = 2n-2n\vert x\vert^2 +2\beta\sum_{i\neq j}\frac{x_i}{x_i-x_j}%
    = 2n - 2 n \vert x \vert^2 + n (n-1) \beta .
  \]
  In particular $\mathrm G(\left\vert\cdot\right|^2)\in L^2(\mu)$. Since $\mu$
  is stationary, we then have $\int \mathrm G \vert x \vert^2 \, \d\mu = 0$,
  and we thus recover~\eqref{eq:isop2}.
  
  Combining~\eqref{eq:isop1} and~\eqref{eq:isop2} we get
  \[
    \int x_i^2 \, \widetilde\mu ( \d x) 
    =  \frac 1n \int \vert x \vert^2 \, \mu ( \d x ) 
    = \frac{\beta}2 + \frac {2-\beta} {2n} , 
  \]
  and 
  \[
    \int x_ix_j\,\widetilde\mu(\d x)
    = \frac1{n(n-1)}\int(x_1+\dotsb+x_n)^2-(x_1^2 + \dotsb + x_n^2)\,\widetilde\mu(\d x) 
    = - \frac \beta {2n} . 
  \]
\end{proof}

\begin{remark}[Mean and covariance of beta Hermite ensembles]
  Let $\mu$ and $\widetilde \mu$ be as in Corollary \ref{co:isotropy}. In
  contrast with the probability measure $\widetilde\mu$, the probability
  measure $\mu$ is log-concave but is not centered, even asymptotically as
  $n\to\infty$, and this is easily seen from $0\not\in D$. Moreover, if
  $X_n=(X_{n,1},\ldots,X_{n,n})\sim\mu$ then the famous Wigner theorem for the
  beta Hermite ensemble, see for instance \cite{MR1781846}, states that almost
  surely and in $L^1$, regardless of the way we choose the common probability
  space,
  \begin{equation}\label{eq:BHEL}
    \frac{1}{n}\sum_{i=1}^n\delta_{X_{n,i}}
    \underset{n\to\infty}{\overset{\mathrm{weak}}{\longrightarrow}}
    \nu_\beta
  \end{equation}
  where
  \begin{align}
    \nu_\beta
    &=\arg\inf_\mu
    \left(
      \int\frac{x^2}{2}\d\mu(x)-\beta\iint\log(x-y)\d\mu(x)\d\mu(y)
    \right)\nonumber\\
    &=\frac{\sqrt{2\beta-x^2}}{\beta\pi}
    \mathbbm{1}_{[-\sqrt{2\beta},\sqrt{2\beta}]}(x)\d x.\label{eq:BHELM} 
  \end{align}
  This follows for instance from a large deviation principle. Moreover it can
  be shown that $X_{n,1}\underset{n\to\infty}{\longrightarrow}-\sqrt{2\beta}$
  and $X_{n,n}\underset{n\to\infty}{\longrightarrow}\sqrt{2\beta}$. This
  suggests in a sense that asymptotically, as $n\to\infty$, the mean is
  supported by the whole interval $[-\sqrt{2\beta},\sqrt{2\beta}]$. It is
  quite natural to ask about the asymptotic shape of the covariance matrix of
  $\mu$. Elements of answer can be found in the work of Gustavsson
  \cite{MR2124079}.
\end{remark}
%

\subsection{Log-concavity and curvature} 

The following Lemma is essentially the key of the proof of Theorem \ref{th:PI}
and Theorem \ref{th:LSI}.

\begin{lemma}[Log-concavity and curvature]\label{le:conc}
  Let $\mu$ be as in \eqref{eq:mun}. Then $U$ is $\rho$-convex. In particular,
  for the beta Hermite ensemble \eqref{eq:BHE}, the potential $U$ is
  $n$-convex, for all $\beta>0$.
\end{lemma}

\begin{proof}
  Recall from \eqref{eq:U} that $U(x)=V(x)+U_W(x)$. Observe that $U_W$ is
  convex as a sum of linear maps composed with the convex function $W$.
  Thus, if $V$ is $\rho$-convex then so is $U$. 
\end{proof}

\subsection{Factorization by projection}

The following factorization lemma is the key of the proof of Theorem
\ref{th:opt}.
Let $u$ be the unit vector of $\R^n$ given by the diagonal direction:
\[
  u=\frac{1}{\sqrt{n}}(1,\dotsc,1)
\]
and let $\pi$ and $\pi^\perp$ be the orthogonal projection onto $\mathbb{R}u$
and $(\mathbb{R}u)^\perp=\{v\in\mathbb{R}^n:\langle v,u\rangle=0\}$.

\begin{lemma}[Factorization by projection]\label{le:factor}
  Let $\mu$ be as in \eqref{eq:mun} and let $X$ be a random vector distributed
  according to $\mu$. Assume that the confinement potential $V$ is quadratic: 
  $V=\rho\left\vert\cdot\right\vert^2/2$ for some $\rho >0$.
  Then $\mu$ has a Gaussian factor in the direction $u$ in the sense that
  $\pi(X)=\langle X,u\rangle u$ and $\pi^\perp(X)$ are independent and
  \[
    \langle X , u \rangle \sim \mathcal{N} \left( 0 , \frac 1\rho \right).
  \] 
  Moreover $\pi^\perp(X)$ has density proportional to $\e^{-U}$ with respect
  to the Lebesgue measure on $(\mathbb{R}u)^\perp$.
\end{lemma}

In the special case of the beta Hermite ensemble \eqref{eq:BHE}, the law of
$\langle X,u\rangle=\mathrm{Trace}(H)/\sqrt n$ is easily seen on the random
matrix model $H$ provided by theorems \ref{th:GUEM} and \ref{th:BHEM}.

An extension of Lemma \ref{le:factor} to higher dimensional gases in
considered is~\cite{simconstr}.

\begin{proof}[Proof of Lemma \ref{le:factor}]
  Since $x=\pi(x)+\pi^\perp(x)$ and $\pi(x)=\langle x,u\rangle u$, we have  
  \[
    |x|^2 = \langle x , u \rangle^2 + |\pi^\perp(x)|^2 .
  \]
  Besides it is easily seen that $U_W (x) = U_W ( \pi^\perp(x) )$ for all $x$, a
  property which comes from the shift invariance of the interaction energy
  $U_W$ along $\mathbb{R}u$. Therefore
  \[
    \e^{ - U (x) } = \e^{ - \rho \langle x,u \rangle ^2 / 2 } 
    \times \e^{ - \rho \vert \pi^\perp (x) \vert^2 / 2 - U_W ( \pi^\perp (x) ) } 
    = \e^{ - \rho \langle x,u \rangle ^2 / 2 } \times \e^{ - U ( \pi^\perp (x) ) } . 
  \] 
  So the density of $X$ is the product of a function of $\langle x ,u\rangle$
  by a function of $\pi^\perp(x)$. 
\end{proof}
The result extends naturally by the same proof to the more general quadratic
case $V=\langle Ax,x\rangle$ where $A$ is a symmetric positive definite $n\times n$ matrix,
provided that the diagonal direction $u$ is an eigenvector of $A$. 
\begin{remark}[Gaussian factor and orthogonal polynomials]
  Let $\mu$ be as in Lemma \ref{le:factor}. Let $H_i$ and $H_j$ be two
  distinct univariate (Hermite) orthogonal polynomials with respect to the
  standard Gaussian law $\mathcal{N}(0,I_n)$. Then it follows from Lemma
  \ref{le:factor} that the symmetric multivariate polynomials
  $H_i(\sqrt{\rho/n}(x_1+\cdots+x_n))$ and
  $H_j(\sqrt{\rho/n}(x_1+\cdots+x_n))$ are orthogonal with respect to $\mu$.
  In particular, when $\rho=n$ and with $H_i(x)=x$ and $H_j(x)=x^2-1$, we get
  that $x_1+\cdots+x_n$ and $(x_1+\cdots+x_n)^2-1$ are orthogonal for $\mu$.
\end{remark}
\section{Proofs} 
\label{se:proofs}

\subsection{Proof of Theorems~\ref{th:PI} and~\ref{th:LSI}}
\label{ss:proof:pilsi}

\begin{proof}[Proof of Theorems~\ref{th:PI} and~\ref{th:LSI}]

  Let us first mention that Theorem~\ref{th:LSI} actually implies
  Theorem~\ref{th:PI}. Indeed it is well-known that applying log-Sobolev to a
  function $f$ of the form $f=1+\epsilon h$ and letting $\epsilon$ tend to $0$
  yields the Poincaré inequality for $h$, with half the constant if the
  log-Sobolev inequality. See for instance~\cite{MR1845806} or
  \cite{MR3155209} for details.
  
  In the discussion below, we call \emph{potential} of a probability measure
  $\mu$ the function $-\log \rho$, where $\rho$ is the density of $\mu$ with
  respect to the Lebesgue measure. In view of Lemma~\ref{le:conc} it is enough
  to prove that a probability measure $\mu$ on $\R^n$ whose potential $U$ is
  $\rho$-convex for some positive $\rho$ satisfies the logarithmic Sobolev
  inequality with constant $2 / \rho$. This is actually a well-known fact. It
  can be seen in various ways which we briefly spell out now. Some of these
  arguments require extra assumptions on $U$, namely that the domain of $U$
  equals $\R^n$ (equivalently $\mu$ has full support) and that $U$ is
  $\mathcal C^2$-smooth on $\R^n$. For this reason we first explain a
  regularization procedure showing that these hypothesis can be added without
  loss of generality.

  \emph{Regularization procedure.} Let $\gamma$ be the Gaussian measure whose
  density is proportional to $\e^{ - \rho \vert x\vert^2/ 2}$ and let $f$ be
  the density of $\mu$ with respect to $\gamma$. Clearly $U$ is $\rho$-convex
  if and only if $\log f$ is concave. Next let $(Q_t)$ be the
  Ornstein--Uhlenbeck semigroup having $\gamma$ as a stationary measure,
  namely for every test function $g$
  \[
    Q_t g ( x ) = \E \left[g \left( \e^{-t} x + \sqrt{ 1-\e^{-2t} } G \right) \right] 
  \]
  where $G \sim \gamma$. Since $\gamma$ is
  reversible for $(Q_t)$ the measure $\mu Q_t$ has density $Q_t f$ with
  respect to $\gamma$. Moreover the semigroup $(Q_t)$ satisfies the following
  property
  \[
    f \text{ log-concave } \quad \Rightarrow \quad Q_t f \text{ log-concave} . 
  \]
  This is indeed an easy consequence of the Prékopa--Leindler inequality,
  see~\eqref{eq:PL} below. As a result the potential $U_t$ of $\mu Q_t$ is
  also $\rho$-convex. Besides $U_t$ is clearly $\mathcal C^\infty$ smooth on
  the whole $\R^n$. Lastly since $\lim_{t\to 0} Q_t f(x) = f(x)$ for almost
  every $x$, we have $\mu P_t \to \mu$ weakly as $t$ tends to $0$. As a
  result, if $\mu P_t$ satisfies log-Sobolev with constant $2/\rho$ for every
  $t$, then so does $\mu$.

  \emph{First proof: The Brascamp--Lieb inequality.} A theorem due to
  Brascamp and Lieb \cite{MR0450480} states that if the potential of $\mu$ is
  smooth and satisfies $\mathrm{Hess}(U)(x)>0$ for all $x\in\mathbb{R}^n$,
  then for any $\mathcal{C}^\infty$ compactly supported test function
  $f:\mathbb{R}^n\to\mathbb{R}$, we have the inequality
  \[
    \mathrm{var}_{\mu} ( f ) 
    \leq \int_{\R^n} %
    \left\langle \mathrm{Hess} (U)^{-1} \nabla f,\nabla f \right\rangle \, \d\mu. 
  \]
  If $U$ is $\rho$-convex then $\mathrm{Hess} (U)^{-1}\leq (1/\rho)I_n$ and we
  obtain
  \[
    \mathrm{var}_{\mu} ( f ) %
    \leq \frac{1}{\rho}\int_{\R^n} |\nabla f|^2 \, \d \mu.
  \]
  The extension of this inequality to all $f\in H^1(\mu)$ follows by
  truncation and regularization. Note that this method only works for
  Poincaré. The Brascamp--Lieb inequality does not seem to admit a logarithmic
  Sobolev inequality counterpart, see \cite{MR1800062} for a discussion.
  
  \emph{Second proof: Caffarelli's contraction theorem.} Again let $\gamma$ be
  the Gaussian measure on $\R^n$ whose density is proportional to
  $\e^{- \rho \vert x \vert^2 /2}$. The theorem of
  Caffarelli~\cite{MR1889232,MR1800860} states that if the potential of $\mu$
  is $\rho$-convex then the Brenier map from $\gamma$ to $\mu$ is
  $1$-Lipschitz. This easily implies that the Poincaré constant of $\mu$
  is at least as good as that of $\gamma$, namely $1/\rho$. Let us sketch the
  argument briefly. Let $T$ be the Brenier map from $\gamma$ to $\mu$ and let
  $f$ be a smooth function on $\R^n$. Using the fact that $T$ pushes forward $\gamma$ to $\mu$, 
  the Poincaré inequality for $\gamma$ and the Lipschitz property of $T$ we get
  \begin{equation}\label{eq:caffarelli}
    \begin{split}
      \mathrm{var}_\mu ( f ) = \mathrm{var}_\gamma ( f \circ T ) 
      & \leq \frac 1 \rho \int_{\R^n} \vert \nabla (f \circ T ) \vert^2 \, \d \gamma\\
      & \leq \frac 1 \rho \int_{\R^n} |\nabla f|^2 \circ T \, \d\gamma \\
      & = \frac 1 \rho \int_{\R^n} |\nabla f|^2  \, \d \mu . 
    \end{split}
  \end{equation}
  This contraction principle works just the same for log-Sobolev. 
 
  \emph{Third proof: The Bakry--Émery criterion.} Assume that $U$ is
  finite and smooth on the whole $\R^n$ and consider the Langevin diffusion
  \[
    \d X_t = \sqrt 2 \, \d B_t - \nabla U (X_t) \, \d t . 
  \]
  The generator of the diffusion is the operator
  $\mathrm{G} = \Delta - \langle \nabla U , \nabla\rangle$. The carré du
  champ $\Gamma$ and its iterated version $\Gamma_2$ are easily computed:
  \begin{equation}\label{eq:BE}
    \begin{split}
      \Gamma (f,g) & %
      = \frac 12 ( \mathrm{G} (fg) - f \mathrm{G}(g) - g \mathrm{G}(f) ) %
      = \langle \nabla f, \nabla g \rangle \\
      \Gamma_2 (f,g) & %
      = \frac 12 ( \mathrm{G} \Gamma (f,g) -  \Gamma (f,\mathrm{G}g) - \Gamma ( \mathrm{G} f, g ) \\
 & = \mathrm{Tr} \left( \mathrm{Hess} (f) \mathrm{Hess} (g) \right) 
      + \langle \mathrm{Hess} (U)\nabla f, \nabla g \rangle .
    \end{split}
  \end{equation}
  We also set $\Gamma ( f ) = \Gamma (f,f)$ and similarly for $\Gamma_2$.  
  The hypothesis that $U$ is $\rho$-convex thus implies that 
  \[
    \Gamma_2 (f) \geq \rho \Gamma (f) ,  
  \] 
  for every suitable $f$. 
  Actually this inequality is equivalent to the condition that $U$ is $\rho$-convex, 
  as can be seen by plugging in linear functions. 
  In the language of Bakry--Émery, see \cite{MR889476,MR1845806,MR3155209},
  the diffusion satisfies the curvature dimension criterion
  $\mathrm{CD} ( \rho , \infty)$. This criterion implies that the stationary
  measure $\mu$ satisfies the following logarithmic Sobolev inequality
  \[
    \mathrm{ent}_\mu ( f^2 ) \leq \frac 2\rho \int_{\R^n} \Gamma (f) \, d\mu , 
  \]
  see~\cite[Proposition~5.7.1]{MR3155209}. Formally this proof also works if
  $\mu$ does not have full support by adding a reflection at the boundary,
  just as in section~\ref{se:dynamics}. However this poses some technical
  issues which are not always easy to overcome. As a matter of fact,
  diffusions with reflecting boundary conditions are not treated in the
  book~\cite{MR3155209}.
  
  \emph{Fourth proof: An argument of Bobkov and Ledoux.} This fourth proof is
  the one that requires the least background. Another nice feature is that the
  regularization procedure is not needed for this proof. It is based on the
  Prékopa-Leindler inequality. The latter, which is a functional form of the
  Brunn--Minkowski inequality, states that if $f,g,h$ are functions on
  $\R^n$ satisfying
  \[
    (1-t) f ( x) + tg(y) \leq h((1-t)x+ty) 
  \]
  for every $x,y\in \R^n$ and for some $t\in [0,1]$, then 
  \begin{equation}\label{eq:PL}
    \left( \int \e^f \, \d x \right)^{1-t} \left( \int \e^g \, \d x\right)^t 
    \leq \int \e^h \, \d x  .  
  \end{equation}
  We refer to~\cite{MR1491097} for a nice presentation of this inequality. 
  Let $F\colon \R^n \to \R$ be a smooth 
  function with compact support, and for $s >0$ 
  let $R_s F$ be the infimum convolution 
  \[ 
    R_s F (x) = \inf_{y\in \R^n}\left\{F(x+y) + \frac 1{2s} \vert y \vert^2\right\}.
  \]
  Fix $t \in (0,1)$. Using the $\rho$-convexity of $U$:
  \[
    (1-t) U (x) + t U(y) \leq U ( (1-t)x+ty ) - \frac { \rho t(1-t) } 2 \vert x-y \vert^2 , 
  \]
  it is easily seen that the functions 
  $f = R_{t/\rho} F - U$, $g=-U$ and $h = (1-t) F-U$ 
  satisfy the hypothesis of the Prékopa-Leindler inequality.  
  The conclusion~\eqref{eq:PL} rewrites in this case 
  \begin{equation}\label{eq:bobkovledoux}
    \left( \int_{\R^n} \e^{ R_{ t / \rho } F} \, \d\mu \right)^{1-t} %
    \leq \int_{\R^n} \e^{ (1-t) F } \, \d \mu .  
  \end{equation}
  It is well-known that $(R_s)$ solves the Hamilton--Jacobi equation
  \[
    \partial R_s F  + \frac 12 \vert \nabla R_s F \vert^2 = 0,    
  \]
  see for instance \cite{MR2597943}. Using this and differentiating the
  inequality~\eqref{eq:bobkovledoux} at $t=0$ yields
  \[
    \mathrm{ent}_\mu ( \e^F )
    \leq \frac 1 {2\rho} \int_{\R^n} \vert \nabla F \vert^2 \e^F \, \d\mu
  \]
  which is equivalent to the desired log-Sobolev inequality. We refer to
  the article~\cite{MR1800062} for more details.
\end{proof}

\subsection{Proof of Theorem~\ref{th:opt} and comments on optimality}
\label{ss:proof:opt}

In the case of the beta ensemble, Theorem~\ref{th:lassalle:generator} shows
that $x\in\mathbb{R}^n\mapsto x_1+\cdots+x_n$ is the only symmetric function
optimal in the Poincaré inequality, up to additive and multiplicative
constants. Our goal now is to study the optimality far beyond this special case.

We have just seen that if a measure $\mu$ has density $\mathrm{e}^{-\phi}$ where
$\phi$ is $\rho$-convex for some $\rho>0$ then it satisfies Poincar\'e with
constant $1/\rho$. This constant is sharp in the case where $\mu$ is the
Gaussian measure whose density is proportional to
$\e^{-\rho \vert x \vert^2 / 2 }$. Indeed the Poincaré constant of that
Gaussian measure is equal to $1/\rho$ and extremal functions are affine
functions, see for instance~\cite{MR1845806,MR3155209}. Similarly, its
log-Sobolev constant is $2/\rho$ and extremal functions are log-affine
functions, see~\cite{MR1132315}. The next lemma asserts that conversely, if
the Poincaré constant of $\mu$ or its log-Sobolev constant matches the bound
predicted by the strict convexity of its potential, then $\mu$ must be
Gaussian in some direction.

Recall the notion of having a Gaussian factor in a given direction, used in
Lemma \ref{le:factor}.
\begin{lemma}\label{le:courtade}
  Let $\mu$ be a probability measure on $\R^n$ with density
  $\mathrm{e}^{-\phi}$ where $\phi$ is $\rho$-convex for some $\rho>0$, and
  assume that there exists a non constant function $f$ such that
  \begin{equation}\label{eq:courtade}
    \mathrm{var}_\mu ( f )
    = \frac 1 \rho \int_{\R^n} \vert \nabla f \vert^2 \,\d\mu .  
  \end{equation}
  Then the following properties hold true: 
  \begin{itemize}
  \item[(i)] The function $f$ is affine: there exists a vector $u$ 
    and a constant $b$ such that
    \[
      f(x) = \langle u ,x\rangle + b.
    \]
  \item[(ii)] The measure $\mu$ has a Gaussian factor of variance $1/\rho$ in the 
    direction $u$.
  \end{itemize}
  Besides, there is a similar statement for the log-Sobolev inequality: if there 
  exists a non constant function $f$ such that 
  \[
    \mathrm{ent}_\mu ( f^2 )
    = \frac 2{\rho}  \int_{\R^n} \vert \nabla f \vert^2 \,\d\mu , 
  \]
  then $\log f$ is affine and $\mu$ has a Gaussian factor in the corresponding direction. 
\end{lemma}
The Poincar\'e case is contained in the main result 
of Cheng and Zhou's article~\cite{MR3575913}. 
The general result is a consequence of the works of de Philippis and
Figalli~\cite{MR3614644}, and also Courtade and Fathi~\cite{CF}.
These authors actually
establish a stability estimate for this lemma: if there exists a function $f$
which is near optimal in Poincar\'e then $\mu$ nearly has a Gaussian factor.
We sketch a proof of Lemma~\ref{le:courtade} based on their ideas. 
\begin{proof}[Proof of Lemma \ref{le:courtade}]
  We analyze the equality case in the third proof of the main theorem, the one
  based on Caffarelli's contraction theorem. Recall that the Brenier map $T$
  from $\gamma$ to $\mu$ is the gradient of a convex function and that it
  pushes forward $\gamma$ to $\mu$. Recall also that Caffarelli's theorem
  asserts that under the hypothesis of the lemma $T$ is $1$-Lipshitz.
  Therefore $T$ is differentiable almost everywhere, and its differential is a
  symmetric matrix satisfying
  \begin{equation}\label{eq:caffa}
    0 \leq (\d T)_x \leq I_n 
  \end{equation}
  as quadratic forms. Now, observe that if $f$ satisfies~\eqref{eq:courtade} then
  every inequality in~\eqref{eq:caffarelli} must actually be an equality. In
  particular $f \circ T$ must be optimal in the Poincaré inequality for
  $\gamma$. This implies that $f\circ T$ is affine. 
  Also there is equality in the inequality
  \[
    \vert \nabla (f \circ T) (x) ) \vert^2  \leq \vert \nabla f (Tx) \vert^2  
  \]
  for almost every $x$. Because of~\eqref{eq:caffa} this actually implies that 
  \[
    (\d T)_x ( \nabla f ( Tx ) ) = \nabla f ( Tx ) ,  
  \]
  for almost every $x$. Since $f\circ T$ is affine the left hand side is constant, 
  and we obtain that $f$ itself must be affine. 
  Thus, there exists a
  vector $u$ and a constant $b$ such that $f(x) = \langle u,x\rangle + b$, and
  moreover $(\d T)_x (u) = u$ for almost every $x$. By a change of variable,
  we can assume that $u$ is a multiple of the first coordinate vector. The
  differential of $T$ at $x$ thus has the form
  \[
    (\d T)_x =
    \left( 
      \begin{matrix} 
        1 & 0 \\ 0 & * 
      \end{matrix} 
    \right)  
  \]
  for almost every $x$. Therefore 
  \[
    T(x_1,\dotsc,x_n)  = ( x_1 + a , S ( x_2 , \dotsc , x_n) )
  \]
  for some constant $a$ and some map $S$ from $\R^{n-1}$ to itself. This
  implies that the image $\mu$ of $\gamma$ by $T$ is a product measure, and
  that the first factor is the Gaussian measure with mean $a$ and variance
  $1/\rho$. This finishes the proof of the first part of the lemma. The
  log-Sobolev version can be obtained very similarly and we omit the details.
\end{proof}

\begin{remark}[Alternate proof based on the Bakry--Émery calculus]\label{rm:optimality} %
  We spell out briefly an alternative proof of Lemma~\ref{le:courtade} based
  on the Bakry--Émery calculus starting from a work by Ledoux
  \cite{MR1160084}. Let $\mathrm{G}$, $\Gamma$, and $\Gamma_2$ be as in
  \eqref{eq:BE}, and let ${(P_t)}_{t\geq0}$ be the Markov semigroup generated
  by $\mathrm{G}$. The usual Bakry--Émery method gives, up to regularity
  considerations,
  \[
    \mathrm{var}_\mu(f) = \frac{1}{\rho}\int\Gamma f\d\mu
    -\frac{2}{\rho}\int_0^\infty
    \left(\int(\Gamma_2-\rho\Gamma)(P_tf)\d\mu\right)\d t . 
  \]
  This is a Taylor--Lagrange formula expressing the ``deficit'' in the
  Poincaré inequality. It shows that if $\Gamma_2 \geq \rho \Gamma$ and
  $\mathrm{var}_\mu ( f ) = \frac{1}{\rho}\int\Gamma f\d\mu$ then
  $(\Gamma_2-\rho\Gamma)(P_tf)(x)=0$ almost everywhere in $t$ and $x$. Up to
  regularity issues, we get in particular
  \[
    (\Gamma_2-\rho\Gamma)(f)=0.
  \]
  In other words, denoting $\left\Vert\cdot\right\Vert_{\mathrm{HS}}$ the
  Hilbert--Schmidt norm,
  \[
    \Vert \mathrm{Hess} (f) \Vert_{\mathrm{HS}}^2 + \langle \left( \mathrm{Hess}
      (U)- \rho I_n \right) \nabla f , \nabla f \rangle = 0 .
  \]
  Since both term are non negative this actually implies that 
  \[
  \mathrm{Hess} (f) = 0 \quad \text{and} \quad \langle \left( \mathrm{Hess} (U) - \rho I_n \right) \nabla f , \nabla f \rangle  = 0 . 
  \]
  Thus $f$ is affine: there exists a unit vector $u$ and two constants
  $\lambda$ and $c$ such that
  \[
    f(x) = \lambda\langle u,x \rangle + c  ,
  \] 
  and moreover $\langle \mathrm{Hess} (U)  u , u \rangle  = \rho$ .
  Since $\mathrm{Hess} (U) \geq \rho I_n$ this actually implies that
  \[
  \mathrm{Hess} (U) u = \rho\,u 
  \]
  pointwise. Proceeding
  as in the proof of Lemma \ref{le:courtade}, we then see that $\mu$ has a Gaussian factor
  of variance $1/\rho$ in the direction $u$. There is a similar argument for
  the log-Sobolev inequality using
  \[
    \mathrm{ent}_\mu(f)
    =\frac{1}{2\rho}\int \Gamma ( \log f ) \, f \,\d \mu
    -\frac{1}{\rho}\int_0^\infty
    \left(\int(\Gamma_2-\rho\Gamma)(\log P_tf) \, P_t f \, \d\mu \right) \,\d t.
  \]
  This leads to the fact that if $f$ is optimal in the logarithmic Sobolev
  inequality then $f$ is of the form $f(x) = \e^{\lambda\langle u,x\rangle+c}$
  and $\mu$ has a Gaussian factor in the direction $u$. As usual, this
  seductive approach requires to justify rigorously the computations via
  delicate handling of regularity and smoothing, see for instance
  \cite{MR3155209}, \cite{MR1842428}, and \cite[Section 4.4.2]{MR2760897}.
\end{remark}

\begin{proof}[Proof of Theorem \ref{th:opt}]
  According to Lemma \ref{le:factor}, if
  $V=\rho\left\vert\cdot\right\vert^2/2$ for some $\rho>0$ then $\mu$ has a
  Gaussian factor in the diagonal direction $u= (1,\dotsc,1) / \sqrt n$. As we
  have seen, this Gaussian satisfies Poincaré with constant $1/\rho$ and
  log-Sobolev with constant $2/\rho$. Moreover, affine functions are optimal
  in Poincaré and log-affine functions are optimal in log Sobolev. This shows
  that we have equality in the Poincar\'e inequality of Theorem~\ref{th:PI}
  for functions $f$ of the form $f (x) = \lambda( x_1 + \dotsb + x_n ) + c$
  for some constants $\lambda$ and $c$, and equality in the logarithmic
  Sobolev inequality of Theorem~\ref{th:LSI} for functions whose logarithm is
  of the preceding type.

  Let us now prove that these are the only optimal functions. Assume that $f$
  is non constant and extremal in the Poincar\'e inequality. Then by
  Lemma~\ref{le:courtade}, there exists a vector $v$ and a constant $b$ such
  that $f(x) = \langle v , x \rangle + b$, and moreover $\mu$ has a Gaussian
  factor in the direction $v$. Since the support of $\mu$ is the set
  $\{ x_1 \geq \dotsb \geq x_n \}$ this can only happen if $v$ is proportional
  to the diagonal direction $u$, which is the result. The proof for
  log-Sobolev is similar.
\end{proof}

\subsection{Proof of Corollary \ref{co:conc} and comments on concentration}
\label{ss:proof:conc}

\begin{proof}[Proof of Corollary \ref{co:conc}]
  The Gaussian concentration can be deduced from the log-Sobolev inequality
  via an argument due to Herbst, see for instance \cite{MR1849347}, which
  consists in using log-Sobolev with $f=\e^F$ to get the Gaussian upper bound
  on the Laplace transform
  \begin{equation}\label{eq:lapgau}
    \int \e^F \, \d\mu %
    \leq \exp \left( \int F \, \d\mu + \frac{\Vert F \Vert_\mathrm{Lip}^2}{2\rho}
    \right), 
  \end{equation}
  which leads in turn to the concentration inequality \eqref{eq:conc} via the
  Markov inequality. Alternatively we can use the intermediate
  inequality~\eqref{eq:bobkovledoux} obtained in the course of the fourth
  proof of Theorem~\ref{th:LSI}. Indeed applying Jensen's inequality to the
  right-hand side of~\eqref{eq:bobkovledoux} and letting $t\to 1$, we obtain
  \begin{equation}\label{eq:lapgau2}
    \int \e^{ R_{1/\rho} F}  \, \d\mu \leq \exp \left( \int F \, \d\mu \right) .  
  \end{equation}
  Moreover, if $F$ is Lipschitz it is easily seen that 
  \[
    R_{1/\rho} F \geq F - \frac 1{2\rho} \Vert F \Vert_\mathrm{Lip}^2 . 
  \]
  Plugging this into the previous inequality yields \eqref{eq:lapgau}. Note
  that a result due to Bobkov and Götze states that \eqref{eq:lapgau2} 
  is equivalent to a Talagrand $\mathrm{W}_2$
  transportation inequality for $\mu$, see for instance \cite{MR1849347} and
  references therein.
  
  In the case $F(x_1,\ldots,x_n)=\frac{1}{n}\sum_{i=1}^nf(x_i)=L_n(f)(x)$ we have
  \[
    \LIP{F}\leq\frac{\LIP{f}}{\sqrt{n}}
  \]
  so that \eqref{eq:conclinstat} follows from \eqref{eq:conc}.
  
  Finally taking $F(x_1,\ldots,x_n)=\max(x_1,\ldots,x_n)$ ($=x_1$ on $D$) in
  \eqref{eq:conc} gives \eqref{eq:concmax}.
\end{proof}

\begin{remark}[Concentration of measure in transportation
  distance]\label{rk:gozlan}
  Following Gozlan \cite{gozlan}, it is possible to obtain concentration of
  measure inequalities in Kantorovich--Wasserstein distance $\mathrm{W}_2$
  from the Hoffman--Wielandt inequality.
  Namely, given a Hermitian matrix $A$, we let $x_1(A)\geq\cdots\geq x_n(A)$
  be the eigenvalues of $A$, arranged in decreasing order, and
  \[
    L_A=\frac{1}{n}\sum_{i=1}^n\delta_{x_i(A)}
  \]
  be the corresponding empirical measure. If $B$ is another Hermitian matrix,
  we get from the Hoffman--Wielandt inequality
  \begin{equation}\label{eq:HWW}
    n\mathrm{W}_2(L_A,L_B)^2=\sum_{i=1}^n(x_i(A)-x_i(B))^2
    \leq\mathrm{Trace}((A-B)^2)=\left\Vert A-B\right\Vert_{\mathrm{HS}}^2.
  \end{equation}
  Thanks to the triangle inequality for $\mathrm{W}_2$, this implies that for
  every probability measure $\nu$ on $\R$ with finite second moment, the
  Lipschitz constant of the map $A\mapsto\mathrm{W}_2(L_A,\mu)$ with respect
  to the Hilbert--Schmidt norm is at most $1/\sqrt{n}$. If $G$ is a Gaussian
  matrix with density proportional to $\e^{-n\mathrm{Trace}(X^2)}$, the Gaussian
  concentration inequality then yields
  \[   
    \mathbb P \left( \left\vert \mathrm{W}_2(L_G,\nu)-\mathbb{E}\mathrm{W}_2(L_G,\nu) \right\vert > r \right)
    \leq2\e^{-\frac{n^2}{2}r^2}.
  \]
  Note that $\nu$ is arbitrary. This inequality can be reformulated as
  follows: If $\mu$ is the Gaussian unitary ensemble in $\R^n$ and $L$ is the
  map $x\in\R^n \mapsto \frac 1n \sum_{i\leq n} \delta_{x_i}$ then for any
  probability measure $\nu$ on $\R$ we have
\[
  \mu \left( \left\vert \mathrm{W}_2(L,\nu)- \int \mathrm{W}_2(L,\nu) \, d\mu  \right\vert > r \right)
  \leq2\e^{-\frac{n^2}{2}r^2}. 
\]
More generally this inequality remains valid when $\mu$ is the law of the eignevalues
  of a random matrix satisfying Gaussian concentration with rate $n$. This is the case for instance if
  the matrix has independent entries satisfying a logarithmic Sobolev
  inequality with constant $1/n$.
\end{remark}

\begin{remark}[Proof for GUE/GOE via Hoffmann--Wielandt inequality]
  For the GUE and the GOE one can give a fifth proof, based on the contraction
  principle, like the proof using Caffarelli's theorem above. The
  Hoffman--Wielandt inequality \cite{MR0052379,MR2978290,MR1477662},
  states that for all $n\times n$ Hermitian matrices $A$ and $B$ with ordered
  eigenvalues $x_1(A)\geq\cdots\geq x_n(A)$ and $x_1(B)\geq\cdots\geq x_n(B)$
  respectively, we have
  \[
    \sum_{i=1}^n(x_i(A)-x_i(B))^2\leq\sum_{i,j=1}^n|A_{ij}-B_{ij}|^2.
  \]
  In other words the map which associates to an $n\times n$ Hermitian matrix
  $A$ its vector of eigenvalues $(x_1(A),\ldots,x_n(A))\in\mathbb{R}^n$ is
  $1$-Lipschitz for the Euclidean structure on $n\times n$ Hermitian matrices,
  given by $\langle A,B\rangle=\mathrm{Trace}(AB)$. On the other hand, as we
  saw in section~\ref{se:GUE}, the Gaussian unitary ensemble is the image by
  this map of the Gaussian measure on $\mathbb H_n$ whose density is
  proportional to $\e^{ -(n/2)\mathrm{Trace}( H^2 )}$. The Poincaré constant
  of this Gaussian measure is $1/n$ so by the contraction principle the
  Poincaré constant of the GUE is $1/n$ at most. The argument works similarly
  for log-Sobolev and for the GOE.
\end{remark}

\subsection{Proof of Theorem \ref{th:equilib}}
\label{ss:proof:equilib}

\begin{proof}[Proof of Theorem~\ref{th:equilib}]
  The exponential decay of relative entropy~\eqref{eq:relent} is a well-known
  consequence of the logarithmic Sobolev inequality, see for
  instance~\cite[Theorem~5.2.1]{MR3155209}. The decay in Wasserstein distance
  follows from the Bakry--Émery machinery, see
  \cite[Theorem~9.7.2]{MR3155209}. Alternatively it can be seen using parallel
  coupling. We explain this argument briefly.

  Let $X$ and $Y$ be two solutions
  of the stochastic differential equation \eqref{eq:SDE} driven by the same
  Brownian motion:
  \[
    \begin{split}
      \d X_t & = \sqrt{2} \, \d B_t - \nabla U (X_t) \, \d t + \d \Phi_t \\
      \d Y_t & =  \sqrt{2} \, \d B_t - \nabla U (Y_t) \, \d t + \d \Psi_t , 
    \end{split}
  \]
  where $\Phi$ and $\Psi$ are the reflections at the boundary of the Weyl
  chamber of $X$ and $Y$ respectively, see section~\ref{se:dynamics} for a
  precise definition. Assume additionally that $X_0 \sim \nu_0$,
  $Y_0 \sim \nu_1$ and that
  \[
    \E(\vert X_0 - Y_0 \vert^p) = \mathrm{W}_p ( \nu_0,\nu_1)^p .  
  \]  
  Observe that
  \[
    \d \vert X_t - Y_t \vert^2 
    = - 2\langle X_t - Y_t , \nabla U ( X_t ) - \nabla U (Y_t) \rangle \, \d t 
    +2\langle X_t-Y_t,\d\Phi_t\rangle + 2\langle Y_t - X_t , \d \Psi_t \rangle. 
  \]
  Since $U$ is $\rho$-convex
  $\langle X_t - Y_t , \nabla U ( X_t ) - \nabla U (Y_t ) \rangle \geq \rho
  \vert X_t - Y_t \vert^2$. Besides $\d\Phi_t = -\mathrm n_t \d L_t$ where $L$
  is the local time of $X$ at the boundary of the Weyl chamber $D$ and
  $\mathrm n_t$ is an outer unit normal at $X_t$. Since $Y_t\in D$ and since
  $D$ is convex we get in particular
  $\langle X_t - Y_t , \d \Phi_t \rangle \leq 0$ for all $t$, and similarly
  $\langle Y_t - X_t , \d \Psi_t \rangle \leq 0$. Thus
  $\d\vert X_t - Y_t \vert^2 \leq - 2\rho\vert X_t - Y_t \vert^2 \, \d t$,
  hence
  \[
    \vert X_t - Y_t \vert \leq \e^{ -\rho t } \vert X_0 - Y_0 \vert . 
  \] 
  Taking the $p$-th power and expectation we get, in $[0,+\infty]$,
  \[
    \E[\vert X_t - Y_t \vert^p ]^{1/p} %
    \leq \e^{-\rho t} \, \E(\vert X_0 - Y_0 \vert^p)^{1/p} %
    = \e^{-\rho t} \mathrm{W}_p( \nu_0 , \nu_1 ).
  \]  
  Moreover since $X_t\sim \nu_0 P_t$ and $Y_t \sim \nu_1 P_t$ we have by
  definition of $\mathrm{W}_p$
  \[
    \mathrm{W}_p( \nu_0 P_t , \nu_1 P_t ) %
    \leq  \E[\vert X_t - Y_t \vert^p ]^{1/p} . 
  \]
  Hence the result.
\end{proof}

\bigskip
 
\textbf{Acknowledgments.} This work is linked with the French research project
ANR-17-CE40-0030 - EFI - Entropy, flows, inequalities. A significant part was
carried out during a stay at the Institute for Computational and Experimental
Research in Mathematics (ICERM), during the 2018 Semester Program on ``Point
Configurations in Geometry, Physics and Computer Science'', thanks to the kind
invitation by Edward Saff and Sylvia Serfaty. We thank also Sergio Andraus,
François Bolley, Nizar Demni, Peter Forrester, Nathaël Gozlan, Michel Ledoux,
Mylène Maïda, and Elizabeth Meckes for useful discussions and some help to
locate references. We would also like to thank the anonymous referee for his 
careful reading of the manuscript and his suggestion to use~\cite{MR3575913,MR3614644,CF}
to prove Lemma~\ref{le:courtade} and Theorem~\ref{th:opt}. 

\bigskip

{\footnotesize This note takes its roots in the blog post \cite{chafai-blog}.}

\bibliographystyle{smfplain}
\bibliography{guefi}

\end{document}